\documentclass[12pt]{article}
\usepackage{amssymb,amsthm,amsmath,latexsym}
\newtheorem{theorem}{Theorem}
\newtheorem{proposition}[theorem]{Proposition}
\newtheorem{lemma}[theorem]{Lemma}
\theoremstyle{remark}
\newtheorem{remark}[theorem]{Remark}
\newcommand{\FF}{\mathbb{F}}
\newcommand{\ZZ}{\mathbb{Z}}

\newcommand{\cG}{\mathcal{G}}
\newcommand{\allone}{\mathbf{1}}
\newcommand{\allzero}{\mathbf{0}}

\DeclareMathOperator{\GL}{GL}
\DeclareMathOperator{\Aut}{Aut}
\DeclareMathOperator{\wt}{wt}

\title{A complete classification of doubly even \\ self-dual codes of 
length 40\footnote{This work was supported by JST PRESTO program.}
}

\author{
Koichi Betsumiya\\
\small Graduate School of Science and Technology\\[-0.8ex] 
\small Hirosaki University\\[-0.8ex] 
\small Hirosaki 036--8561, Japan\\
\small\tt betsumi@cc.hirosaki-u.ac.jp\\
\and
Masaaki Harada\\
\small Department of Mathematical Sciences\\[-0.8ex] 
\small Yamagata University\\[-0.8ex] 
\small Yamagata 990--8560, Japan, and\\[-0.8ex] 
\small PRESTO, Japan Science and Technology Agency (JST)\\[-0.8ex] 
\small Saitama 332--0012, Japan \\[-0.8ex] 
\small\tt mharada@sci.kj.yamagata-u.ac.jp\\
\and
Akihiro Munemasa\\
\small Graduate School of Information Sciences\\[-0.8ex] 
\small Tohoku University\\[-0.8ex] 
\small Sendai 980--8579, Japan\\
\small\tt munemasa@math.is.tohoku.ac.jp\\
}

\begin{document}

\maketitle

\begin{abstract}
A complete classification of binary doubly even
self-dual codes of length $40$ is given.
As a consequence, a classification of binary extremal
self-dual codes of length $38$ is also given.

%\bigskip\noindent \textbf{Keywords:} 
% self-dual code; weight enumerator; mass formula
\end{abstract}

%%%%%%%%%%%%%%%%%%%%%%%%%%%%%%
\section{Introduction}

As described in~\cite{RS-Handbook},
self-dual codes are an important class of linear codes for both
theoretical and practical reasons.
It is a fundamental problem to classify self-dual codes
of modest lengths and 
much work has been done towards classifying self-dual codes over $\FF_q$
for $q=2$ and $3$,
where $\FF_q$ denotes the finite field of order $q$
and $q$ is a prime power
(see~\cite{RS-Handbook}).

Codes over $\FF_2$ are called {\em binary} and
%% all codes in this paper are binary unless otherwise noted.
all codes in this paper are binary.
The \textit{dual code} $C^{\perp}$ of a code 
$C$ of length $n$ is defined as
$
C^{\perp}=
\{x \in \FF_2^n \mid x \cdot y = 0 \text{ for all } y \in C\},
$
where $x \cdot y$ is the standard inner product.
A code $C$ is called 
% \textit{self-orthogonal} if $C \subset C^{\perp}$, 
% and $C$ is  called 
\textit{self-dual} if $C = C^{\perp}$. 
A self-dual code $C$ is {\em doubly even} if all
codewords of $C$ have weight divisible by four, and {\em
singly even} if there is at least one codeword of weight $\equiv 2
\pmod 4$.
It is  known that a self-dual code of length $n$ exists 
if and only if  $n$ is even, and
a doubly even self-dual code of length $n$
exists if and only if $n$ is divisible by eight.
The minimum weight $d$ of a self-dual code of length $n$
is bounded by
$d  \le 4 \lfloor{\frac {n}{24}} \rfloor + 6$ 
if $n \equiv 22 \pmod {24}$, 
$d  \le 4  \lfloor{\frac {n}{24}} \rfloor + 4$ 
otherwise \cite{MS73} and \cite{Rains}.
A self-dual code meeting the bound is called  {\em extremal}.

Two codes $C$ and $C'$ are {\em equivalent}, denoted $C \cong C'$,
%% Two codes are {\em equivalent} 
if one can be
obtained from the other by permuting the coordinates.
An {\em automorphism} of $C$ is a permutation of the coordinates of $C$
which preserves $C$.
The set consisting of all automorphisms of $C$ is called the
{\em automorphism group} of $C$ and it is denoted by 
$\Aut(C)$. 

A classification of doubly even self-dual codes 
was done for lengths $8,16$ in \cite{Pless72},
for length $24$ in \cite{PS75} and for length $32$ in \cite{CPS}.
For length $40$, only some partial classifications have been
done by various authors. Extremal doubly even self-dual
codes of length $40$ with automorphism of a prime order
$p$ having $c$ cycles have been classified for
$(p,c)=(19,2),(7,5),(5,4)$ in \cite{Yo},
$(p,c)=(3,6)$ in \cite{B04},
$(p,c)=(3,8)$ in \cite{Kim10}, and
$(p,c)=(5,8)$ in \cite{YZ}.
The main aim of this paper is to give a classification
of doubly even self-dual codes of length $40$.

\begin{theorem}\label{thm:main}
There are $94343$ inequivalent doubly even self-dual codes of 
length $40$, $16470$ of which are extremal.
\end{theorem}
% 77873+16470=94343

As a summary, we list in Table~\ref{Tab:N}
the total number $N_T(n)$ of 
inequivalent doubly even self-dual codes of length $n$ and
the number $N_d(n)$ of 
inequivalent doubly even self-dual codes of length 
$n$ ($n=8,16,\ldots,40$)
and minimum weight $d$ ($d=4,8$).

%%%%%%%%%%%%%%%%%%%%%%%%%%%%%%%%%%%%%%%%%%%%%%%%
\begin{table}[th]
\caption{Number of doubly even self-dual codes}
\label{Tab:N}
\begin{center}
{\small
%{\footnotesize
%{\scriptsize
\begin{tabular}{c|c|c|c}
\noalign{\hrule height0.8pt}
Length $n$ & \multicolumn{1}{c|}{$N_T(n)$}  
& \multicolumn{1}{c|}{$N_4(n)$}  & \multicolumn{1}{c}{$N_8(n)$}  \\
\hline
 8 & 1    & 1    & - \\
16 & 2    & 2    & - \\
24 & 9    & 8    & 1 \\
32 & 85   & 80   & 5 \\
40 & 94343& 77873& 16470 \\
\noalign{\hrule height0.8pt}
  \end{tabular}
}
\end{center}
\end{table}
%%%%%%%%%%%%%%%%%%%%%%%%%%%%%%%%%%%%%%%%%%%%%%%%

A classification of singly even
self-dual codes of lengths up to $36$ 
is known 
\cite{B06},
\cite{BR02},
\cite{CPS},
\cite{HM36},
\cite{Pless72}, 
\cite{PS75}.
As a consequence of Theorem \ref{thm:main}, 
we give a classification of extremal
singly even self-dual codes of length $38$.

Generator matrices of all inequivalent doubly even self-dual codes
of length $40$ and 
extremal self-dual codes of length $38$ 
can be obtained electronically from~\cite{Data}.
All computer calculations in this paper
were done by {\sc Magma}~\cite{Magma}.

%%%%%%%%%%%%%%%%%%%%%%%%%%%%%%%%%%%%%%
\section{Classification method}\label{Sec:main}

In this section,
we describe how to complete a classification of
doubly even self-dual codes of length $40$.

%%%%%
\subsection{Preliminaries}
The weight enumerator of a doubly even self-dual code
of length $40$ can be written as:
\begin{multline}
1 
+ A_4 y^4 
+(285 + 24 A_4) y^8 
+(21280  + 92 A_4) y^{12} 
\\ 
+(239970  - 600 A_4) y^{16} 
+(525504  + 966 A_4) y^{20} 
+ \cdots + y^{40},
\label{we}
\end{multline}
where $A_w$ denotes the number of codewords of weight $w$
(see e.g.~\cite{MS73}).

The number of  distinct doubly even self-dual
codes of length $n$ is given \cite{MST} by the formula: 
\begin{equation}
\prod_{i=0}^{n/2-2}(2^{i}+1).
\end{equation}
King \cite{King} determined the 
number of  distinct extremal doubly even self-dual
codes of length $40$.
Let $N(40,d)$ denote the number of distinct doubly even self-dual
codes of length $40$ and minimum weight $d$ $(d=4,8)$.
Then we have
\begin{align*}
%% &N(40,4)+N(40,8) = \prod_{i=0}^{18}(2^{i}+1) \text{ and, }\\
&N(40,4)=
4009357722800739726876619952910304312989584368968750,\\
%%2 \cdot 3^{13}  5^6 11^2 13^2 17^2 19 \cdot 29 \cdot
%%37 \cdot 667921207135875391513765573687, \\
&N(40,8) 
= 10263335567003567415076803513287627980544163840000000.
%%= 2^{20}  3^{14}  5^{ 7}  7^{ 4} 11^{ 4} 13^{ 3} 17^{ 2} 
%%19 \cdot 23 \cdot 29 \cdot 31 \cdot 37 \cdot
%%9433 \cdot 8558891707.
\end{align*}
%% > Factorization(10263335567003567415076803513287627980544163840000000);
%% [ <2, 20>, <3, 14>, <5, 7>, <7, 4>, <11, 4>, <13, 3>, <17, 2>, <19, 1>, 
%% <23, 1>, <29, 1>, <31, 1>, <37, 1>, <9433, 1>, <8558891707, 1> ]
%% > Factorization(4009357722800739726876619952910304312989584368968750);
%% [ <2, 1>, <3, 13>, <5, 6>, <11, 2>, <13, 2>, <17, 2>, <19, 1>, 
%% <29, 1>, <37, 1>, <667921207135875391513765573687, 1> ]

%%%%%
\subsection{Minimum weight 4}

Let $C$ be a singly even self-dual code and
let $C_0$ denote the
subcode of codewords having weight $\equiv0\pmod4$.
Then $C_0$ is a subcode of codimension $1$.
The {\em shadow} $S$ of $C$ is defined to be 
$C_0^\perp \setminus C$ \cite{C-S}.
There are cosets $C_1,C_2,C_3$ of $C_0$ such that
$C_0^\perp = C_0 \cup C_1 \cup C_2 \cup C_3 $, where
$C = C_0  \cup C_2$ and $S = C_1 \cup C_3$.

\begin{proposition}[Brualdi and Pless \cite{BP}]\label{prop:BP}
Let $C$ be a self-dual code of length $n \equiv 6 \pmod 8$.
Let $C_0,C_1,C_2$ and $C_3$ be as above.
Then
\begin{multline*}
C^* =
\{(v,0,0)\mid v\in C_{0}\}\cup
\{(v,1,1)\mid v\in C_2\}
\\
\quad\cup
\{(v,1,0)\mid v\in C_{1}\}\cup
\{(v,0,1)\mid v\in C_{3}\}
\end{multline*}
is a doubly even self-dual code of length $n+2$.
\end{proposition}

%% Starting from known self-dual codes of length $38$,
%% one can generate a number of doubly even self-dual codes of
%% length $40$.

There are $519492$ inequivalent self-dual codes of 
length $36$ \cite{HM36}.
By considering the direct sum of 
the unique self-dual code of length $2$ and 
each of these codes,
we have $519492$ self-dual codes of length $38$ and minimum weight $2$.
By Proposition \ref{prop:BP},
$519492$ doubly even self-dual codes of length $40$
and minimum weight $4$ are constructed.

% There are $519492$ inequivalent self-dual codes of 
% length $36$ \cite{HM36}.
% Since any self-dual code of length $n+2$ and minimum weight $2$
% is decomposable as $i_2\oplus C_{n}$,
% where $i_2$ is the unique self-dual code of length $2$
% and $C_{n}$ is some self-dual code of length $n$,
% we have $519492$ inequivalent self-dual codes
% of length $38$ and minimum weight $2$.
% By Proposition \ref{prop:BP},
% $519492$ doubly even self-dual codes of length $40$
% and minimum weight $4$ are constructed.

We examine the equivalence
or inequivalence of codes as follows.
Let $C$ be a doubly even self-dual code of length $40$
and minimum weight $d$ ($d=4,8$).
Let $M(C)$ be the $A_8 \times 40$ matrix with rows composed of the 
codewords of weight $8$ 
in $C$, where the $(1,0)$-matrix $M(C)$ is regarded as a matrix over $\ZZ$.
We define
\[
N(C)=
\left\{\begin{array}{ll}
\{n_{ij} \mid 1\le i, j \le 40\} \setminus \{57\}
& \text{ if $C$ is extremal,}\\
\{n_{ij} \mid 1\le i, j \le 40\} & \text{ otherwise,}
\end{array}\right.
\]
where $n_{ij}$ is the $(i,j)$-entry of $M(C)^T M(C)$,
and $M(C)^T$ denotes the transposed matrix of $M(C)$.
% Here $M(C)^T$ denotes the transpose of $M(C)$.
The codewords of weight $w$ in $C$
are calculated by the {\sc Magma} function {\tt Words}. 
%% Note that the codewords of weight $8$ in an extremal
%% doubly even self-dual code form a $1$-$(40,8,57)$ design.
Note that the codewords of weight $8$ in $C$
form a $1$-$(40,8,57)$ design when $C$ is extremal.
This means that $n_{ii}=57$ for any $i$ $(i=1,2,\ldots,40)$
and $\max \{n_{ij} \mid 1\le i, j \le 40\}=57$
when $C$ is extremal.
Then we consider the following:
\[
\alpha(C)=(
\#\Aut(C),
A_4,
\max N(C),
\min N(C),
\#N(C)).
\]
The automorphism group $\Aut(C)$ 
of the code $C$
is calculated by the {\sc Magma} function {\tt AutomorphismGroup}. 
Of course, $C$ and $C'$ are inequivalent if $\alpha(C)\ne \alpha(C')$.
For a given set of codes, we divided into classes
where each class contains codes $C$ with identical $\alpha(C)$.  
Then we divided the codes in each class
into equivalence classes.
This was done by the {\sc Magma} function {\tt IsIsomorphic}.

In this way, 
we checked equivalences among the above
$519492$ doubly even self-dual codes of length $40$
and minimum weight $4$. %% constructed by Proposition \ref{prop:BP}.
Then we obtained the set 
$\mathcal{C}_{40,4}$ of $77873$
inequivalent doubly even self-dual codes with minimum weight $4$
satisfying
\begin{equation}\label{eq:d4}
\sum_{C \in \mathcal{C}_{40,4}}
\frac{40!}{\#\Aut(C)} = N(40,4).
\end{equation}
This shows that there is no other doubly even self-dual code
of length $40$ and minimum weight $4$.
The numbers $N(A_4)$ 
of doubly even self-dual codes of length $40$
containing $A_4$ codewords of weight $4$
are listed in Table~\ref{Tab:40}.

%%%%%%%%%%%%%%%%%%%%%%%%%%%%%%%%%%%%%%%%%%%%%%%%
\begin{table}[thb]
\caption{Number of doubly even self-dual codes of length $40$}
\label{Tab:40}
\begin{center}
{\small
%{\footnotesize
%{\scriptsize
\begin{tabular}{c|c|c|c|c}
\noalign{\hrule height0.8pt}
\multicolumn{5}{c}{$(A_4,N(A_4))$}\\
\hline
$( 0,16470)$&$(13, 382)$&$(26, 47)$&$(40, 12)$&$( 64,  3)$\\
$( 1,20034)$&$(14, 374)$&$(27, 16)$&$(41,  1)$&$( 66,  1)$\\
$( 2,17276)$&$(15, 231)$&$(28, 38)$&$(42,  9)$&$( 70,  3)$\\
$( 3,12168)$&$(16, 236)$&$(29, 13)$&$(43,  3)$&$( 72,  1)$\\
$( 4, 8471)$&$(17, 143)$&$(30, 29)$&$(44,  7)$&$( 74,  1)$\\
$( 5, 5552)$&$(18, 160)$&$(31,  7)$&$(46,  7)$&$( 78,  1)$\\
$( 6, 3916)$&$(19, 100)$&$(32, 22)$&$(48,  4)$&$( 90,  1)$\\
$( 7, 2610)$&$(20, 104)$&$(33,  3)$&$(50,  4)$&$( 92,  1)$\\
$( 8, 1932)$&$(21,  54)$&$(34, 25)$&$(52,  6)$&$( 94,  2)$\\
$( 9, 1243)$&$(22,  90)$&$(35,  3)$&$(54,  2)$&$(106,  1)$\\
$(10, 1093)$&$(23,  37)$&$(36, 11)$&$(56,  1)$&$(134,  1)$\\
$(11,  669)$&$(24,  59)$&$(37,  4)$&$(58,  4)$&$(190,  1)$\\
$(12,  605)$&$(25,  26)$&$(38, 11)$&$(62,  2)$& \\
%  0&16470& 13&  382& 26&   47& 40&   12& 64&    3\\
%  1&20034& 14&  374& 27&   16& 41&    1& 66&    1\\
%  2&17276& 15&  231& 28&   38& 42&    9& 70&    3\\
%  3&12168& 16&  236& 29&   13& 43&    3& 72&    1\\
%  4& 8471& 17&  143& 30&   29& 44&    7& 74&    1\\
%  5& 5552& 18&  160& 31&    7& 46&    7& 78&    1\\
%  6& 3916& 19&  100& 32&   22& 48&    4& 90&    1\\
%  7& 2610& 20&  104& 33&    3& 50&    4& 92&    1\\
%  8& 1932& 21&   54& 34&   25& 52&    6& 94&    2\\
%  9& 1243& 22&   90& 35&    3& 54&    2&106&    1\\
% 10& 1093& 23&   37& 36&   11& 56&    1&134&    1\\
% 11&  669& 24&   59& 37&    4& 58&    4&190&    1\\
% 12&  605& 25&   26& 38&   11& 62&    2&&\\
\noalign{\hrule height0.8pt}
  \end{tabular}
}
\end{center}
\end{table}
%%%%%%%%%%%%%%%%%%%%%%%%%%%%%%%%%%%%%%%%%%%%%%%%

%\clearpage
\subsection{Minimum weight 8}\label{subsec:mw8}

For a set of coordinates $I\subset\{1,2,\ldots,n\}$,
let $\pi:\FF_2^n\rightarrow\FF_2^t$,
$\pi':\FF_2^n\rightarrow\FF_2^{n-t}$
be the projection to the set of coordinates $I$,
$I'$, respectively,
where $I'=\{1,\ldots,n\}\setminus I$ and $\# I=t$.
For a code $C$ of length $n$, the {\em punctured code} and 
the {\em shortened code} of $C$ on the set of coordinates $I$
are the codes
$\pi'(C)$ and $\{\pi'(c)\mid c\in C, \; \pi(c)=\allzero\}$, 
respectively, where $\allzero$ denotes the zero vector.

If $C$ is a doubly even code containing the all-one vector
$\allone$, then
we denote by $q_C:C^\perp/C\to\FF_2$ the map defined by
$q_C(x+C)=\frac{\wt(x)}{2}\bmod2$, where $\wt(x)$ denotes
the weight of $x$. It is easy to verify that
the map $q_C$ is well-defined.

% Let $C_i$ be a doubly even $[n_i,k_i]$ code
% containing $\allone$, for $i=1,2$.
Let $C_i$ be a doubly even code of length $n_i$
containing $\allone$, for $i=1,2$.
A bijective linear map
\begin{equation}\label{eq:isometric}
f:C_1^\perp/C_1\to C_2^\perp/C_2
\end{equation}
is called an {\em isometry} if
$q_{C_1}=q_{C_2}\circ f$.
The set of isometries (\ref{eq:isometric}) is denoted by
$\Phi(C_1,C_2)$. Note that an isometry 
% exists only if $n_1-n_2=2(k_1-k_2)$ and $n_1\equiv n_2
exists only if $n_1-n_2=2(\dim C_1-\dim C_2)$ and $n_1\equiv n_2
\pmod8$.
For an isometry $f\in\Phi(C_1,C_2)$, we define a code
\begin{equation}\label{eq:consta}
D(C_1, C_2, f)=\{(x_1,x_2)\mid x_1\in C_1^\perp,\;x_2\in f(x_1+C_1)\}
\subset\FF_2^{n_1+n_2},
\end{equation}
% where $(x_1|x_2)$ denote the concatenation of the vectors
% $x_1,x_2$.
It is easy to see that $D(C_1, C_2, f)$ is a doubly
even self-dual code.
Conversely, every doubly even self-dual code of length
$n_1+n_2$ containing a codeword of weight $n_1$ can be
constructed by this method. Indeed, let $x$ be a codeword
of weight $n_1$ in a doubly even self-dual code $C$ of
length $n_1+n_2$. Let $C_2$ (resp.\ $C_1$)
be the shortened code of $C$ on the support (resp.\
the complement of the support) of $x$. Then $C_1$
and $C_2$ are doubly even codes. Moreover, $C_2^\perp$
(resp.\ $C_1^\perp$)
is the punctured code of $C$ on the support (resp.\
the complement of the support) of $x$.
Let $\pi$ and $\pi'$ denote
the projections onto the support of $x$ and the complement
of the support of $x$, respectively.
We define $f:C_1^\perp/C_1\to C_2^\perp/C_2$ by
$f(x_1+C_1)=\pi'(x)+C_2$, where $x$ is a codeword
of $C$ satisfying $\pi(x)=x_1$. Then
$D(C_1,C_2,f)$ is equivalent to $C$.

For fixed codes $C_1$, $C_2$, the resulting code
$D(C_1, C_2, f)$
depends on the choice of an isometry $f$.
However, some of these codes are equivalent to each other.
We will give a sufficient condition
for two resulting codes to be equivalent.
We need this criterion to reduce the amount of 
calculation to be reasonable.

First, we define some groups.
For a doubly even code $C$ containing $\allone$,
we denote by  $\cG_0(C)$ the subgroup of $\GL(C^\perp/C)$
induced by the action of $\Aut(C)$ on the linear space $C^\perp/C$ and
denote by $\cG_1(C)$ the subgroup 
$\Phi(C,C)$ of $\GL(C^\perp/C)$.
By the definition, the group $\cG_0(C)$ is a subgroup of $\cG_1(C)$.
If we replace $f$ by $\sigma_2\circ f\circ \sigma_1$,
where $\sigma_i\in\cG_0(C_i)$, then the resulting codes
are equivalent, that is,
\[
D(C_1, C_2, f)\cong D(C_1, C_2, \sigma_2\circ f\circ \sigma_1).
\]
This means that, in order to enumerate the set of codes
$\{D(C_1, C_2, h)\mid h\in\Phi(C_1, C_2)\}$
up to equivalence,
we first fix $f\in\Phi(C_1, C_2)$, and
it suffices to enumerate the codes
$D(C_1, C_2, f \circ g)$,
where $g$ runs through a set of representatives for the double cosets
\[
(f^{-1}\circ\cG_0(C_2)\circ f)\backslash\cG_1(C_1)/\cG_0(C_1).
\]

We now apply this method with $(n_1,n_2)=(16,24)$
in order to classify extremal doubly even self-dual codes of length $40$.
Note that from the weight enumerator (\ref{we}),
such a code has a codeword of weight $16$.
This means that every extremal doubly even self-dual code of length $40$
is equivalent to $D(C_1,C_2,f)$ for some
doubly even code $C_1$ of length $16$ containing $\allone$, 
some doubly even code $C_2$ of length $24$ containing $\allone$, 
and $f\in\Phi(C_1,C_2)$.
%%The dimension of $C_1$ is restricted.
All doubly even codes of lengths $16$ and $24$ can be found
in \cite{Miller}.

However, if $\dim C_1\leq2$, then the degree of
$\cG_1(C_1)\subset\GL(C_1^\perp/C_1)$
as a permutation group is too large to perform the double coset
enumeration,
so we only enumerated codes $D(C_1,C_2,f)$, where
%% $C_1$ is a doubly even $[16,k]$ code with $k\geq3$. In this way,
$C_1$ is a doubly even code of length $16$ with $\dim C_1\geq3$.
Here, the group $\cG_1(C_1)$ was
constructed by the {\sc Magma} function {\tt GOPlus}, and
the double coset enumeration was  performed using the
{\sc Magma} function {\tt DoubleCosetRepresentatives}.
% However, if $\dim C_1\leq2$, then the degree of
% $\cG_1(C_1)\subset\GL(C_1^\perp/C_1)$
% as a permutation group is too large to perform the double coset
% enumeration using the {\sc Magma} function 
% {\tt DoubleCosetRepresentatives},
% so we only enumerated codes $D(C_1,C_2,f)$, where
% %% $C_1$ is a doubly even $[16,k]$ code with $k\geq3$. In this way,
% $C_1$ is a doubly even code of length $16$ with $\dim C_1\geq3$.
Then we classified the resulting codes using
the method described in the previous subsection.
In this way,
we obtained a set of pairwise inequivalent $16468$
extremal doubly even self-dual codes of length $40$. 
It turns out that there
are two other codes. One is the code with automorphism group
%of order $6840$ constructed by Yorgov \cite{Yo}.
of order $6840$ constructed in \cite{Yo}.
The other is the code 
$H^{(1234)}B'_6$ in the notation of \cite{YZ}
and this code has automorphism group of order $120$.
In this way,
%By ????,
we obtained the set 
$\mathcal{C}_{40,8}$ of $16470$
inequivalent extremal doubly even self-dual codes
satisfying
\begin{equation}\label{eq:d8}
\sum_{C \in \mathcal{C}_{40,8}}
\frac{40!}{\#\Aut(C)} = N(40,8).
\end{equation}
From (\ref{eq:d4}) and (\ref{eq:d8}),
it follows that there is no other doubly even self-dual code
of length $40$.
This explains the number $N(0)$
of extremal doubly even self-dual codes of length $40$
listed in Table~\ref{Tab:40}.
Therefore, we have Theorem \ref{thm:main}.

\section{Some properties}
In this section, we give some properties of doubly even self-dual
codes of length $40$.

%%%%%
%% \subsection{Covering radii of doubly even self-dual codes}
%%% CR
The covering radius of a code $C$ of length $n$ 
is the smallest integer $R$ such that spheres of radius $R$ 
around codewords of $C$ cover the space $\FF_2^n$.
It is known that 
the covering radius is the same as the largest value
among weights of cosets.
Here, the weight of a coset is the smallest weight of 
a vector in the coset.
The covering radius is a basic and important geometric parameter of 
a code.
Assmus and Pless \cite{A-P} began the study of the covering radii of 
(extremal) doubly even self-dual codes.

Let $R_{40,d}$ be the covering radius of a doubly 
even self-dual code of length $40$ and minimum weight $d$
$(d=4,8)$.
Then, by the sphere-covering bound and the Delsarte 
bound (see \cite{A-P}),
$6 \le R_{40,8} \le 8$ and $6 \le R_{40,4} \le 10$.
In Table \ref{Tab:CR},
we list the numbers $N(d,R)$ of 
doubly even self-dual codes with minimum weight $d$ and
covering radius $R$.
This was calculated by the {\sc Magma} function {\tt CoveringRadius}. 
From the above calculation, 
we have the following:

\begin{proposition}
There is no doubly even self-dual code of length $40$ 
with covering radius $6$.
\end{proposition}

\begin{remark}
In \cite{HO}, based on a preprint by Michio Ozeki,
the non-existence of an extremal doubly even self-dual code with
covering radius $6$ was announced.
However, unfortunately, his preprint contained an error and,
in his paper \cite{Ozeki-S} he withdrew the above announcement.
From the above calculation, the non-existence of an extremal 
doubly even self-dual code with
covering radius $6$ was verified.
\end{remark}

\begin{remark}
The two extremal doubly even self-dual 
codes with covering radius $7$ can be found 
in \cite{HMT} and \cite{HO}.
\end{remark}

%%%%%%%%%%%%%%%%%%%%%%%%%%%%%%%%%%%%%%%%%%%%%%%%
\begin{table}[tbh]
\caption{Covering radii of doubly even self-dual codes}
\label{Tab:CR}
\begin{center}
{\small
%{\footnotesize
%{\scriptsize
\begin{tabular}{c|c|c}
\noalign{\hrule height0.8pt}
% $R$ & 6 & 7 &     8 &   9 &  10 \\
% \hline
% $N(8,R)$ & 0 &  2& 16468 & - & - \\
% $N(4,R)$ & 0 & 23& 76768 & 954 & 128 \\
$R$   &$N(4,R)$ & $N(8,R)$ \\
\hline
 6 &    0 &       0 \\
 7 &   23 &       2 \\
 8 &76768 &   16468 \\
 9 &  954 &       - \\
10 &  128 &       - \\
\noalign{\hrule height0.8pt}
  \end{tabular}
}
\end{center}
\end{table}
Now we give some properties of extremal doubly even self-dual
codes of length $40$.
%%%%%%%%%%%%
Let $\sigma$ be an automorphism of odd prime order $p$.
If $\sigma$ has $c$ independent $p$-cycles and $f$ fixed points,
then $\sigma$ is said to be of type $p$-$(c,f)$.
All extremal doubly even self-dual codes of length $40$ 
with automorphism of type $p$-$(c,f)$ are known for
$p \ge 5$ (see \cite[Table 3]{Huffman05}).
The cases with $(p,c)=(3,6)$ and $(3,8)$ were considered 
in \cite{B04} and \cite{Kim10}, respectively.
The numbers $N(p,c)$ of inequivalent extremal doubly
even self-dual codes with automorphism of type $p$-$(c,f)$
are listed in Table \ref{Tab:pcf}
for $(p,c)=(3,6), (3,10)$ and $(3,12)$.
It is claimed in \cite[Theorem 12]{B04} that $N(3,6)=16$.
However, we verified that $N(3,6)=17$.
Since the list of the $16$ codes is not available,
we are unable to compare the result with ours.

\begin{proposition}
There are $17$, $70$ and $322$
inequivalent extremal doubly even self-dual codes of length $40$ 
with automorphism of types
$3$-$(6,22)$, $3$-$(10,10)$ and $3$-$(12,4)$, respectively.
\end{proposition}

%%%%%%%%%%%%%%%%%%%%%%%%%%%%%%%%%%%%%%%%%%%%%%%%
\begin{table}[thbp]
\caption{$N(p,c)$ for $(p,c)=(3,6), (3,10)$ and $(3,12)$}
\label{Tab:pcf}
\begin{center}
{\small
%{\footnotesize
%{\scriptsize
\begin{tabular}{c|c|c|c}
\noalign{\hrule height0.8pt}
$(p,c)$ & $(3,6)$ & $(3,10)$ & $(3,12)$ \\
\hline
$N(p,c)$ & 17 & 70 & 322 \\
\noalign{\hrule height0.8pt}
  \end{tabular}
}
\end{center}
\end{table}
%%%%%%%%%%%%%%%%%%%%%%%%%%%%%%%%%%%%%%%%%%%%%%%%
In Table \ref{Tab:Aut}, we list the numbers $N(\#\Aut)$ of
extremal doubly even self-dual codes with automorphism groups of
order $\#\Aut$.

%%%%%%%%%%%%%%%%%%%%%%%%%%%%%%%%%%%%%%%%%%%%%%%%
\begin{table}[thb]
\caption{Orders of automorphism groups}
\label{Tab:Aut}
\begin{center}
{\small
%{\footnotesize
%{\scriptsize
\begin{tabular}{c|c|c|c|c}
\noalign{\hrule height0.8pt}
\multicolumn{5}{c}{$(\#\Aut,N(\#\Aut))$}\\
\hline
$(1,10400)$&$(36,1   )$&$(256,21 )$&$(3072,3 )$&$(61440,1)$   \\
$(2,3538 )$&$(38,1   )$&$(288,4  )$&$(3840,1 )$&$(65536,1)$   \\
$(3,43   )$&$(40,5   )$&$(320,1  )$&$(4096,1 )$&$(110592,1)$  \\
$(4,1189 )$&$(48,34  )$&$(384,12 )$&$(4608,2 )$&$(147456,1)$  \\
$(5,2    )$&$(60,2   )$&$(512,16 )$&$(5376,1 )$&$(245760,1)$  \\
$(6,68   )$&$(64,75  )$&$(576,3  )$&$(6144,7 )$&$(737280,1)$  \\
$(8,459  )$&$(72,4   )$&$(720,2  )$&$(6840,1 )$&$(786432,1)$  \\
$(10,8   )$&$(96,12  )$&$(768,7  )$&$(9216,1 )$&$(983040,1)$  \\
$(12,80  )$&$(114,1  )$&$(1024,3 )$&$(12288,2)$&$(1474560,1)$ \\
$(16,233 )$&$(120,5  )$&$(1296,1 )$&$(16384,1)$&$(5505024,1)$ \\
$(18,1   )$&$(128,46 )$&$(1536,10)$&$(18432,1)$&$(8257536,1)$ \\
$(20,4   )$&$(144,4  )$&$(1728,1 )$&$(20480,1)$&$(44236800,1)$\\
$(24,41  )$&$(160,1  )$&$(1920,1 )$&$(20736,1)$&$(82575360,1)$\\
$(30,2   )$&$(192,12 )$&$(2048,4 )$&$(32768,1)$& \\
$(32,70  )$&$(240,2  )$&$(2688,1 )$&$(49152,3)$& \\
\noalign{\hrule height0.8pt}
  \end{tabular}
}
\end{center}
\end{table}
%%%%%%%%%%%%%%%%%%%%%%%%%%%%%%%%%%%%%%%%%%%%%%%%
% 1&10400&36&1   &256&21 &3072&3 &61440&1   \\
% 2&3538 &38&1   &288&4  &3840&1 &65536&1   \\
% 3&43   &40&5   &320&1  &4096&1 &110592&1  \\
% 4&1189 &48&34  &384&12 &4608&2 &147456&1  \\
% 5&2    &60&2   &512&16 &5376&1 &245760&1  \\
% 6&68   &64&75  &576&3  &6144&7 &737280&1  \\
% 8&459  &72&4   &720&2  &6840&1 &786432&1  \\
% 10&8   &96&12  &768&7  &9216&1 &983040&1  \\
% 12&80  &114&1  &1024&3 &12288&2&1474560&1 \\
% 16&233 &120&5  &1296&1 &16384&1&5505024&1 \\
% 18&1   &128&46 &1536&10&18432&1&8257536&1 \\
% 20&4   &144&4  &1728&1 &20480&1&44236800&1\\
% 24&41  &160&1  &1920&1 &20736&1&82575360&1\\
% 30&2   &192&12 &2048&4 &32768&1& \\
% 32&70  &240&2  &2688&1 &49152&3& \\

% %%%%%%%%%%%%%%%%%%%%%%%%%%%% Put here %%%%%%%%%%%%%%%%%%%%%%%%%%%%
% %%%%%%%%%%%%%%%%%%%%%%%%%%%% Put here %%%%%%%%%%%%%%%%%%%%%%%%%%%%
% %\verb+ ========== Put here ========== +
% %Put some observation of the special two codes

As we mentioned at the end of Subsection~\ref{subsec:mw8},
we have the following:

\begin{proposition}\label{prop:dimCx}
Let $C_x$ denote the shortened code of $C$ on the complement of
the support of a codeword $x$. Then there are two inequivalent
extremal doubly even self-dual codes $C$ of length $40$ 
such that $\dim C_x\leq 2$ for all $x\in C$ with $\wt(x)=16$.
\end{proposition}

Although the condition given in Proposition~\ref{prop:dimCx} 
can be characterized
by the vanishing of a coefficient in the weight enumerator of
genus $3$ (see \cite{Runge}), we have not been able to prove
Proposition~\ref{prop:dimCx} directly, without classifying all extremal doubly even
self-dual codes of length $40$.

%\verb+ ========== Put here ========== +
%%%%%%%%%%%%%%%%%%%%%%%%%%%% Put here %%%%%%%%%%%%%%%%%%%%%%%%%%%%
%%%%%%%%%%%%%%%%%%%%%%%%%%%% Put here %%%%%%%%%%%%%%%%%%%%%%%%%%%%

In Table \ref{Tab:rank}, we list the numbers $N(\dim)$ of 
extremal doubly even self-dual codes such that subcodes
generated by codewords of weight $8$
have dimension $\dim$.
The dimension is the same as the $2$-rank of
the $1$-$(40,8,57)$ design formed by
the codewords of weight $8$.

%%%%%%%%%%%%%%%%%%%%%%%%%%%%%%%%%%%%%%%%%%%%%%%%
\begin{table}[thbp]
\caption{Dimensions of subcodes generated by codewords of weight $8$}
\label{Tab:rank}
\begin{center}
{\small
%{\footnotesize
%{\scriptsize
\begin{tabular}{c|c|c|c|c}
\noalign{\hrule height0.8pt}
$\dim$ & 17 & 18 & 19 & 20 \\
\hline
$N(\dim)$ & 5 & 1 & 14& 16450 \\
\noalign{\hrule height0.8pt}
  \end{tabular}
}
\end{center}
\end{table}
%%%%%%%%%%%%%%%%%%%%%%%%%%%%%%%%%%%%%%%%%%%%%%%%

%%%%%%%%%%%%%%%%%%%%%%%%%%%%%%%%%%%%%%
\section{Extremal self-dual codes of length 38}

Let $D$ be a doubly even self-dual code of length $40$.
Let $C$ be the code obtained from $D$ for which
some particular pair of coordinates $i,j$ are $00$
and $11$ and deleting these coordinates.
Then $C$ is a self-dual code of length $38$.
Here, we say that $C$ is obtained from $D$ by subtracting coordinates
$i,j$.
In addition, 
any self-dual code of length $38$ is obtained from
some doubly even self-dual code of length $40$ 
by subtracting some two coordinates
(see \cite{CPS}).
%% due to computer time limitations,
Due to the computational complexity,
we only completed a classification of extremal
self-dual codes of length $38$.
Note that there are at least $13644433$
inequivalent self-dual codes of
length $38$ \cite{HM36}.

%% Let $C$ be a singly even self-dual code and
%% let $C_0$ denote the
%% subcode of codewords having weight $\equiv0\pmod4$.
%% Then $C_0$ is a subcode of codimension $1$.
%% The {\em shadow} $S$ of $C$ is defined to be 
%% $C_0^\perp \setminus C$ \cite{C-S}.
%% There are cosets $C_1,C_2,C_3$ of $C_0$ such that
%% $C_0^\perp = C_0 \cup C_1 \cup C_2 \cup C_3 $, where
%% $C = C_0  \cup C_2$ and $S = C_1 \cup C_3$.
Any extremal self-dual code $C$ of length $38$ and
its shadow $S$ have one of the following weight enumerators 
\cite{C-S}:
\begin{align}
\label{eq:WE1}
&\left\{\begin{array}{cl}
W_C=&
1 + 171 y^8 + 1862 y^{10} + 10374 y^{12} + 36765 y^{14} + 84759 y^{16} 
\\ &
+ 128212 y^{18} + \cdots,
\\
W_S=&
114 y^7 + 9044 y^{11} + 118446 y^{15} + 269080 y^{19} + \cdots,
\\
\end{array}\right.
\\
&
\label{eq:WE2}
\left\{\begin{array}{cl}
W_C=&
1 + 203 y^8 + 1702 y^{10} + 10598 y^{12} + 36925 y^{14} + 84055 y^{16} 
\\&
+ 128660 y^{18} + \cdots,
\\
W_S=&
y^3 + 106 y^7 + 9072 y^{11} + 118390 y^{15} + 269150 y^{19} +\cdots.
\\
\end{array}\right.
\end{align}

Although the following two lemmas are somewhat trivial, 
it is useful in finding extremal self-dual codes of length $38$.

\begin{lemma}
Any extremal self-dual code of length $38$ with
weight enumerator (\ref{eq:WE1}) (resp.\ (\ref{eq:WE2}))
is obtained from
some extremal doubly even self-dual code of length $40$ 
(resp.\ some doubly even self-dual code of length $40$ 
containing one codeword of weight $4$)
by subtracting some two coordinates.
\end{lemma}
%%%%%
\begin{proof}
Let $C$ be an  extremal self-dual code of length $38$ with
weight enumerator (\ref{eq:WE1}) (resp.\ (\ref{eq:WE2})).
%% By Theorem~1 in \cite{BP},
By Proposition \ref{prop:BP},
%% \begin{multline*}
%% C^* =
%% \{(v,0,0)\mid v\in C_{0}\}\cup
%% \{(v,1,1)\mid v\in C_2\}
%% \\
%% \quad\cup
%% \{(v,1,0)\mid v\in C_{1}\}\cup
%% \{(v,0,1)\mid v\in C_{3}\}
%% \end{multline*}
a doubly even self-dual code $C^*$ of length $40$ is
constructed.
In addition, by (\ref{eq:WE1}) (resp.\ (\ref{eq:WE2})),
$C^*$ is extremal (resp.\ $C^*$ contains one codeword of
weight $4$).
The code $C$ is obtained from $C^*$ by subtracting the last two coordinates.
The result follows.
\end{proof}

For the remainder of this section, we suppose that 
$D$ is either an extremal 
doubly even self-dual code of length $40$
or 
a doubly even self-dual code of length $40$
containing one codeword of weight $4$.
Also, let $D_{i,j}$ denote the self-dual code of length $38$ 
obtained  from $D$ by subtracting two coordinates $i,j$.

\begin{lemma}\label{lem:M}
Let $M(D)$ be the matrix with rows  composed of the codewords of weight $8$ 
in $D$, where the $(1,0)$-matrix $M(D)$ is regarded as a matrix over $\ZZ$.
\begin{itemize}
\item[\rm (1)]
Suppose that $D$ is extremal.
Then the $(i,j)$-entry of $M(D)^T M(D)$ is zero if and only if
$D_{i,j}$ is extremal.
\item[\rm (2)]
Suppose that $D$ 
%has 
contains
one codeword $x$ of weight $4$.
Then the $(i,j)$-entry of $M(D)^T M(D)$ is zero and
the pair of coordinates $i,j$ in $x$ are $10$ or $01$
if and only if $D_{i,j}$ is extremal.
\end{itemize}
\end{lemma}
\begin{proof}
There is a codeword of weight $8$ in $D$ for which
the coordinates $i,j$ are $11$ if and only if
$D_{i,j}$ contains a codeword of weight $6$.
Suppose that $D$ contains one codeword $x$ of weight $4$.
The coordinates $i,j$ in $x$ are $11$ (resp.\ $00$) if and only if
$D_{i,j}$ contains a codeword of weight $2$ (resp.\ $4$).
\end{proof}

By Lemma \ref{lem:M},
from all inequivalent extremal doubly even self-dual 
codes and all inequivalent doubly even self-dual codes containing 
one codeword of weight $4$, 
we constructed extremal self-dual codes of length $38$ 
which need be checked further for equivalences.
Then we checked equivalences among these codes using
the method similar to that given in Section \ref{Sec:main}.
Finally, we have the following:

\begin{proposition}
There are $2744$ inequivalent extremal self-dual codes
of length $38$.
Of these $1730$ have weight enumerator (\ref{eq:WE1})
and $1014$ have weight enumerator (\ref{eq:WE2}).
%% There are $1730$ inequivalent extremal self-dual codes
%% of length $38$ with weight enumerator (\ref{eq:WE1}).
%% There are $1014$ inequivalent extremal self-dual codes
%% of length $38$ with weight enumerator (\ref{eq:WE2}).
\end{proposition}

\begin{remark}
A classification of extremal self-dual codes of 
length $38$ was very recently obtained in \cite{AGKSS} by somewhat
different techniques.
This was indicated by Jon-Lark Kim
in a private communication \cite{Kim}.
\end{remark}

In Table \ref{Tab:38}, we list the numbers $N(\#\Aut)$ of
extremal self-dual codes with automorphism groups of
order $\#\Aut$ for both weight enumerators
(\ref{eq:WE1}) and (\ref{eq:WE2}).

%%%%%%%%%%%%%%%%%%%%%%%%%%%%%%%%%%%%%%%%%%%%%%%%
\begin{table}[th]
\caption{Number of extremal self-dual codes of length $38$}
\label{Tab:38}
\begin{center}
{\small
%{\footnotesize
%{\scriptsize
\begin{tabular}{c|c|c|c|c}
\noalign{\hrule height0.8pt}
\multicolumn{5}{c}{$(\#\Aut,N(\#\Aut))$} \\
\hline
\multicolumn{5}{c}{Weight enumerator (\ref{eq:WE1})} \\
\hline
$(1,1480)$ &$(4,30)$ &$( 9,1)$ &$( 24,4)$ &$(342,1)$ \\
$(2, 177)$ &$(6, 7)$ &$(12,5)$ &$( 36,1)$ & \\
$(3,  15)$ &$(8, 7)$ &$(18,1)$ &$(168,1)$ & \\
%   1 & 1480 &&\\
%   2 &  177 &&\\
%   3 &   15 &&\\
%   4 &   30 &&\\
%   6 &    7 &&\\
%   8 &    7 &&\\
%   9 &    1 &&\\
%  12 &    5 &&\\
%  18 &    1 &&\\
%  24 &    4 &&\\
%  36 &    1 &&\\
% 168 &    1 &&\\
% 342 &    1 &&\\
\hline
\multicolumn{5}{c}{Weight enumerator (\ref{eq:WE2})}\\
\hline
$( 1, 773)$&$( 4,  38)$&$(12,   3)$&$(24,  10)$&$(216,  1)$\\
$( 2, 145)$&$( 6,  10)$&$(14,   1)$&$(144,  1)$&$(504,  1)$\\
$( 3,  21)$&$( 8,   8)$&$(21,   1)$&$(168,  1)$&\\
\noalign{\hrule height0.8pt}
  \end{tabular}
}
\end{center}
\end{table}
%%%%%%%%%%%%%%%%%%%%%%%%%%%%%%%%%%%%%%%%%%%%%%%%

%%%%%%%%%%%%%%%%%%%%%%%%%%%%%%%%%%%%%%
\bigskip \noindent
{\bf Some historical comments (July 27, 2012).} 
A classification of extremal self-dual codes of 
length $38$ was completed in  \cite{AGKSS}, and
a classification of all self-dual codes of length $38$
was completed in \cite{BB38}.
The paper \cite{AGKSS} was submitted before this paper 
was submitted, and the paper \cite{BB38} was submitted
after this paper was submitted.

%%%%%%%%%%%%%%%%%%%%%%%%%%%%%%%%%%%%%%
%\bigskip
%\noindent
%{\bf Acknowledgment.} 
\subsection*{Acknowledgements}
The authors would like to thank Jon-Lark Kim
for providing information on \cite{AGKSS}.

%%%%%%%%%%%%%%%%%%%  References  %%%%%%%%%%%%%%%%%%%%%%%%

%%%%%%%%%%%%%%%%%%%%%%%%%%%%%%%%%


\begin{thebibliography}{99}
\bibitem{AGKSS}
C.~Aguilar-Melchor, P.~Gaborit, J.-L.~Kim, L.~Sok and P.~Sol\'e.
\newblock
Classification of extremal and $s$-extremal binary self-dual codes 
of length 38.
\newblock
{\em IEEE\ Trans.\ Inform.\ Theory},
58:2253--2262, 2012.

\bibitem{A-P} E.~F. Assmus, Jr.\ and V.~Pless.
\newblock
On the covering radius of extremal self-dual codes.
\newblock
{\em IEEE\ Trans.\ Inform.\ Theory},
29:359--363, 1983.

\bibitem{B06} R.~T. Bilous.
\newblock
Enumeration of the binary self-dual codes of length 34.
\newblock
{\em  J. Combin.\ Math.\ Combin.\ Comput.},
59:173--211, 2006.

\bibitem{BR02} R.~T. Bilous and G.~H.~J. van Rees.
\newblock
An enumeration of self-dual codes of length 32.
\newblock
{\em Des.\ Codes, Cryptogr.},
26:61--86, 2002.

\bibitem{Magma} W.~Bosma and J.~Cannon.
\newblock
Handbook of Magma Functions.
\newblock
Department of Mathematics, University of Sydney,
Available online at 
{\verb+ http://magma.maths.usyd.edu.au/magma/+}.

\bibitem{B04} S.~Bouyuklieva. 
\newblock
Some optimal self-orthogonal and self-dual codes. 
\newblock
{\em Disc.\ Math.},
287:1--10, 2004.
	
\bibitem{BB38}S.~Bouyuklieva and I.~Bouyukliev.
\newblock
An algorithm for classification of binary self-dual codes.
\newblock
{\em IEEE\ Trans.\ Inform.\ Theory},
58:3933--3940, 2012.

\bibitem{BP} R.~Brualdi and V.~Pless.
\newblock
Weight enumerators of self-dual codes.
\newblock
{\em IEEE Trans.\ Inform.\ Theory},
37:1222--1225, 1991.

\bibitem{CPS} J.~H. Conway, V.~Pless and N.~J.~A. Sloane.
\newblock
The binary self-dual codes of length up to $32$: a revised enumeration.
\newblock
{\em J.\ Combin. Theory Ser.~A},
60:183--195, 1992.

\bibitem{C-S} J.~H. Conway and N.~J.~A. Sloane.
\newblock
A new upper bound on the minimal distance of self-dual codes.
\newblock
{\em IEEE\ Trans.\ Inform.\ Theory},
36:1319--1333, 1990.

% \bibitem{HK} M. Harada and H. Kimura,
% New extremal doubly-even $[64, 32,12]$ codes,
% {\sl Des.\ Codes Cryptogr.}
% 6} (1995),  91--96.

\bibitem{HM36} M.~Harada and A.~Munemasa.
\newblock
Classification of self-dual codes of length $36$.
\newblock
{\em Advances Math.\ Communications},
6:229--235, 2012.

\bibitem{Data} M.~Harada and A.~Munemasa.
\newblock
Database of Self-Dual Codes.
\newblock
Available online at
{\verb+ http://www.math.is.tohoku.ac.jp/~munemasa/selfdualcodes.htm+}.
%{\verb+ selfdualcodes.htm+}.

\bibitem{HMT}M.~Harada, A.~Munemasa and K.~Tanabe. 
\newblock
Extremal self-dual $[40,20,8]$ codes with covering radius 7.
\newblock
{\em Finite Fields Appl.},
10:183--197, 2004.

\bibitem{HO}M.~Harada and M.~Ozeki. 
\newblock
Extremal self-dual codes with the smallest covering radius.
\newblock
{\em Disc.\ Math.},
215:271--281, 2000.

\bibitem{Huffman05}W.~C. Huffman. 
On the classification and enumeration of self-dual codes.
{\em Finite Fields Appl.},
11:451--490, 2005.

\bibitem{Kim10}H.~J. Kim. 
\newblock
Self-dual codes with automorphism of order 3 having 8 cycles.
\newblock
{\em Des.\ Codes Cryptogr.},
57:329--346, 2010.

\bibitem{Kim} J.-L. Kim. 
\newblock
private communication. 
\newblock
April 25, 2011.

\bibitem{King}O.~D. King. 
\newblock
The mass of extremal doubly-even self-dual codes of length 40.
\newblock
{\em IEEE Trans.\ Inform.\ Theory},
47:2558--2560, 2001.

\bibitem{MST}F.~J. MacWilliams, N.~J.~A. Sloane and J.~G. Thompson. 
\newblock
Good self dual codes exist.
\newblock
{\em Disc.\ Math.},
3:153--162, 1972.

\bibitem{MS73} C.~L. Mallows and N.~J.~A. Sloane.
\newblock
An upper bound for self-dual codes.
\newblock
{\em Inform.\ Control},
22:188--200, 1973.

\bibitem{Miller} R.~L. Miller.
\newblock
Doubly Even Codes.
\newblock
Available online at
{\verb+http://www.rlmiller.org/de_codes/+}.

\bibitem{Ozeki-S} M.~Ozeki.
\newblock
Jacobi polynomials for singly even self-dual codes and the
covering radius problems.
\newblock
{\em IEEE\ Trans.\ Inform.\ Theory},
48:547--557, 2002.

\bibitem{Pless72}V.~Pless. 
\newblock
A classification of self-orthogonal codes over ${\rm GF}(2)$.
\newblock
{\em Disc.\ Math.},
3:209--246, 1972.

\bibitem{PS75}V.~Pless and N.~J.~A. Sloane. 
\newblock
On the classification and enumeration of self-dual codes.
\newblock
{\em J.\ Combin. Theory Ser.~A},
18:313--335, 1975.

\bibitem{Rains} E.~M. Rains.
\newblock
Shadow bounds for self-dual codes.
\newblock
{\em IEEE Trans.\ Inform.\ Theory},
44:134--139, 1998.

\bibitem{RS-Handbook} E.~Rains and N.~J.~A. Sloane.
\newblock
Self-dual codes.
\newblock
In {\em Handbook of Coding Theory},
V.~S. Pless and W.~C. Huffman (Editors),
pages 177--294, Elsevier, Amsterdam, 1998.

\bibitem{Runge}B.~Runge. 
\newblock
Codes and Siegel modular forms. 
\newblock
{\em Disc.\ Math.},
148:175--204, 1996.

\bibitem{Yo} V.~Y. Yorgov.
\newblock
Binary self-dual codes with automorphisms of odd order.
\newblock
{\em Problems Inform.\ Transmission},
19:260--270, 1984. 
% 1983??
translated from 
{\em Problemy Peredachi Informatsii},
19:11--24, 1983 (Russian).

\bibitem{YZ}V.~Y. Yorgov and N.~Zyapkov. 
\newblock
Doubly even self-dual $[40, 20,8]$-codes with an automorphism 
of odd order.
\newblock
{\em Problems Inform.\ Transmission},
32:253--257, 1997.
% 1996??
translated from 
{\em Problemy Peredachi Informatsii},
32:41--46, 1996 (Russian).
\end{thebibliography}
\end{document}